\documentclass{amsart}
\usepackage{amsfonts}
\usepackage{graphicx}
\usepackage{amscd}
\usepackage{amsmath}
\usepackage{amssymb}
\usepackage{xcolor}

\makeatletter
\@namedef{subjclassname@2010}{%
	\textup{2010} Mathematics Subject Classification}
\makeatother

\usepackage{mathrsfs}
\usepackage{amsbsy}

\newtheorem*{acknowledgement}{Acknowledgement}

\newtheorem{corollary}{\bf Corollary}
\newtheorem{definition}{\bf Definition}
\newtheorem{lemma}{\bf Lemma}

\newtheorem{remark}{Remark}

\newtheorem{theorem}{\bf Theorem}

\theoremstyle{definition}
\newtheorem{example}{\bf Example}

\numberwithin{equation}{section}

\makeatletter
\@namedef{subjclassname@2020}{%
	\textup{2020} Mathematics Subject Classification}
\makeatother

\title[Shrinking Ricci solitons]{Rigidity results for shrinking and expanding Ricci solitons}

\author{B. Leandro}
\author{J. Poveda}

\address{Universidade Federal de Goi\'as, IME, Caixa Postal 131, CEP 74690-900, Goi\^ania, GO, Brazil.}
\email{bleandroneto@ufg.br}
\email{jeferson.arley@discente.ufg.br}

\keywords{gradient shrinking Ricci soliton; expanding gradient Ricci soliton; Ricci flow; curvature estimates} \subjclass[2020]{Primary 53C25, 53C20, 53E20}
\date{\today}

\begin{document}

	\begin{abstract}
		In this paper, we prove some rigidity results for both shrinking and expanding Ricci solitons. First, we prove that compact shrinking Ricci solitons are Einstein if we control the maximum value of the potential function. Then, we prove some rigidity results for non-compact gradient expanding and shrinking Ricci solitons with pinched Ricci and scalar curvatures, assuming an asymptotic condition on the scalar curvature at infinity. 
	\end{abstract}
	
	\maketitle

	\section{Introduction and Main Results}\label{int}
	
	Natural generalizations of Einstein manifolds are Ricci solitons. They are the fixed points of the Ricci flow in the space of Riemannian metrics, and they correspond to self-similar Ricci flow solutions. Ricci solitons played an important role in Perelman's proof of the Poincar\'e conjecture.

	\begin{definition}\label{def1}
		An $n$-dimensional Riemannian manifold $M^n$ with complete Riemannian metric $g$ is a gradient Ricci soliton $(M^n,\,g,\,f,\,\rho)$ if there is a smooth function $f$ on $M$ and a constant $\rho$ such that
		\begin{equation*}
			Ric+\nabla^2\,f=\rho g.
		\end{equation*} Here, $\nabla^2 f$ denotes the Hessian of $f.$ The function $f$ is called potential function. Also, if  $\rho>0$, $\rho<0$ or $\rho=0$, we call gradient soliton, shrinking, expanding, or steady, respectively. 
	\end{definition}
	
	A gradient Ricci soliton is an Einstein manifold when $f$ is constant (trivial). Thus, Ricci solitons are natural extensions of Einstein metrics.

	Hamilton showed that any $2$-dimensional compact shrinking Ricci soliton must be Einstein. In \cite{Ivey}, the author proved that three dimensional compact shrinking Ricci soliton must be Einstein. Any compact steady and expanding Ricci solitons, as well as the compact shrinking Ricci solitons in dimensions two and three, must be trivial. However, in general this is not true. In fact, $\mathbb{CP}^2\#(-\mathbb{CP}^2)$ is a compact non-Einstein shrinking Ricci soliton (cf. \cite{Cao1996,Koiso}).

	In \cite{Hamilton2}, Hamilton conjectured that any compact gradient shrinking Ricci soliton with a positive curvature operator must be Einstein. In \cite{cao2007}, the authors proved that a compact gradient shrinking Ricci soliton must be Einstein, if it admits a Riemannian metric with a positive curvature operator and satisfies an integral inequality. In \cite{Fernandez2010}, Fern\'andez-L\'opez and García-Río gave an important classification for compact shrinking Ricci solitons in terms of the range of the potential function (Theorem 2.4). The works \cite{cao2013,Ni} accurately classified the non-compact case.

	A lot of progress on the classification of shrinking Ricci Soliton has been made during the last few years. In \cite{cao2013}, Cao and Chen provided a classification for complete Bach-flat gradient shrinking Ricci solitons. In \cite[Theorem 1.4]{cao2013} was proved that if the $D$-tensor (cf. Equation \ref{Dtensor}) is identically zero, then a complete four dimensional shrinking Ricci soliton is Einstein, or a finite quotient of $\mathbb{R}^{4}$ or $\mathbb{S}^{3}\times\mathbb{R} $. Moreover, considering $n\geq 5$ the shrinking soliton must be either Einstein, a finite quotient of the Gaussian shrinking soliton $\mathbb{R}^{n}$ or a finite quotient of $N^{n-1}\times\mathbb{R}$, where $N^{n-1}$ is Einstein (see also \cite{Cation}). 
	
	Robinson in \cite{Robinson} used an important technique to prove the uniqueness of static vacuum black holes, where the author provided a very important divergence formula for the static vacuum equations. In \cite{brendle2011}, Brendle based on the work of Robinson got a classification of steady Ricci solitons, but he did not follow closely the procedure of Robinson to construct the divergence formula for the steady Ricci soliton (cf. \cite{leandro2022}). 
	In \cite{benedito2022}, combining \cite{brendle2011} and \cite{Robinson} the authors proved that an $n$-dimensional gradient Yamabe soliton must be a Riemannian manifold with constant scalar curvature.

	In this work, we will build a divergence formula (Lemma \ref{lema22}) for the shrinking Ricci soliton following closely the ideas of Robinson and Brendle, so with the help of this formula we classify a compact shrinking Ricci soliton. 
	
	If a Ricci soliton is trivial (i.e., Einstein), then the scalar curvature 
	satisfies $R = n\rho.$ The following result was inspired also by \cite{Fernandez}.

	Now, we will show  that triviality of compact gradient shrinking Ricci
	solitons can be characterized by a gap on the potential function, see also the pinched conditions assumed by Catino in \cite{Cation}.

	\begin{theorem}
		\label{thmB} Let $(M^n,\,g,\,f,\,\rho)$ be a compact gradient shrinking Ricci soliton such that
		\begin{eqnarray*}
			f(x)\leq \sqrt{\left(\frac{n+2}{4}-\frac{\delta}{4\rho}\right)^2+\frac{n\delta}{4\rho}} - \left(\frac{n+2}{4} - \frac{\delta}{4\rho}\right),
		\end{eqnarray*}
		where $\delta=\min_M(R)$ is a constant. Then, $(M^n,\,g,\,f)$ is Einstein. 
	\end{theorem}

	For any trivial (i.e., Einstein) shrinking Ricci soliton with constant potential function we have $\rho=\frac{R}{n}$. Thus, if $(M^n,\,g,\,f,\,\rho)$ is a compact gradient shrinking Ricci soliton such that $n\rho=\min_MR$ the condition over $f$ in Theorem \ref{thmB} is equivalent to
	\begin{eqnarray*}
		f(x)\leq \frac{1}{2}(\sqrt{n^2+1}-1).
	\end{eqnarray*}
	Hence, the condition over $f$ in Theorem \ref{thmB} is trivially satisfied for an Einstein manifold with, for instance, $f=0$.

	\begin{remark}
		It is well-known there exists a non-Einstein compact shrinking Ricci soliton on $\mathbb{CP}2\#(-\mathbb{CP}2)$. This example do not satisfies the hypothesis assumed in Theorem \ref{thmB}, i.e., the maximum of $f$ in $\mathbb{CP}2\#(-\mathbb{CP}2)$ is bigger that the bound assumed for the potential function in the above theorem.
	\end{remark}
	
	In \cite[Theorem 6.1]{Munteanu2015}, the authors proved an upper bound for the diameter of a compact shrinking Ricci soliton. This bound depends only on the dimension and the injectivity radius of the manifold. Moreover, in \cite{Fernandez} they proved that a compact shrinking Ricci soliton has diameter bounded below by a constant depending of the geometry of the manifold.

	It is important to highlight that the scalar curvature $R$ on a shrinking Ricci soliton is nonnegative and
	\begin{eqnarray}\label{cao2010}
		\frac{1}{4}(r(x)-c_1)^2\leq f(x)\leq \frac{1}{4}(r(x)+c_2)^2,
	\end{eqnarray}
	where $r(x)= d(x_0,\,x)$ is the distance function from some fixed point
	$x_0\in M$, and $c_1$ and $c_2$ are positive constants depending only on $n$ and
	the geometry of $g_{ij}$ on the unit ball $B_{x_0}(1)$ (cf. \cite{cao2010}).
	
	For expanding Ricci solitons we have that on a complete noncompact gradient expanding soliton with nonnegative Ricci curvature (or pinched Ricci curvature) the potential function $f$ satisfies the estimates
	\begin{eqnarray}\label{expandf}
		\frac{1}{4}(r(x)-c_1)^2-c_2\leq -f(x)\leq \frac{1}{4}(r(x)+2\sqrt{-f(x_0)})^2.
	\end{eqnarray}
	Here, $c_1>0$ and $c_2>0$ are constants.
	
	Of course, the normalizing of $f$ and the coefficients in the above estimates has to be adjusted accordingly.

	\
	
	In what follows, we remember some examples of shrinking and expanding Ricci solitons. More examples can be found in \cite{tesechan}.

	\begin{example}[Einstein manifolds]
		The first and trivial examples of gradient Ricci solitons are Einstein manifolds. By choosing $f$ to be a constant function, they are endorsed with soliton structure.
	\end{example}
	
	\begin{example}[Gaussian solitons]
		The flat metric on $\mathbb{R}^n$ with potential function $f=\frac{\rho}{2}|x|^2$, where $\rho$ is either positive or negative (i.e., either Shrinker or Expander).
	\end{example}
	
	\begin{example}[Cylinders]\label{cylinders}
		\
		\begin{enumerate}
			\item[Shrinker:] Consider $M^n=\mathbb{R}^{n-k}\times\mathbb{S}^k$, $k\geq2$, $(x,\,y)\in\mathbb{R}^{n-k}\times\mathbb{S}^k$, the potential function $f(x,\,y)= \frac{(k-1)}{2}|x|^2$ and $\rho=(k-1)$. This soliton have positive constant scalar curvature.
			\item[Expander:] Consider $M^n=\mathbb{R}^{n-k}\times\mathbb{H}^k$, $(x,\,y)\in\mathbb{R}^{n-k}\times\mathbb{H}^k$, the potential function $f(x,\,y)= \frac{-(k-1)}{2}|x|^2$ and $\rho=-(k-1)$. This soliton have negative constant scalar curvature.
		\end{enumerate}
	\end{example}
	
	Now we prove a rigidity result for complete shrinking assuming the curvature is pinched and an asymptotic condition on the gradient of the scalar curvature at infinity holds. Rigidity results for shrinking Ricci solitons assuming pinched conditions was proved by Catino in \cite{Cation,catino2016}.
	
	\begin{theorem}\label{expan}
		Let $(M^n,\,g,\,f,\,\rho)$ be a complete gradient shrinking Ricci soliton such that
		\begin{eqnarray*}
			|\nabla R|=o(r^{-n-2}),
		\end{eqnarray*}
		and one of the following conditions hold:
		\begin{itemize}
			\item[(I)]  $Ric\leq\dfrac{R\left(\frac{2R}{n-1}-\rho\right)}{\left(2\rho(1+f)-\frac{n-3}{n-1}R\right)}.$\\
			\item[(II)] \begin{eqnarray*}
				|\nabla R|&\leq& \frac{(n-1)}{(3n-2)}\left[|\nabla f|\sqrt{\left(2\rho(1+f)-\frac{(n-3)}{(n-1)}R\right)^2+\frac{4(3n-2)}{(n-1)}R\left(\rho-\frac{R}{n-1}\right)}\right]\\
				&+&\frac{(n-1)}{(3n-2)}|\nabla f|\left(2\rho(1+f)-\frac{(n-3)}{(n-1)}R\right).
			\end{eqnarray*} 
		\end{itemize}
		Then, the $D$-tensor must be identically zero. Consequently,
		\begin{itemize}
			\item[(i)] $(M^4,\,g)$ is either Einstein, or a finite quotient of $\mathbb{R}^4$ or $\mathbb{S}^3\times\mathbb{R}$; 
			\item[(ii)] $(M^n,\,g)$, $n \geq 5$, is either Einstein, or is a finite quotient of the Gaussian shrinking soliton, or is a finite quotient of $N^{n-1}\times\mathbb{R}$, where $N^{n-1}$ is Einstein.
		\end{itemize}
	\end{theorem}

	\begin{remark}
		\
		\begin{enumerate}
			\item 
			In the shrinking case we can infer that $$0<\frac{2R}{n-1}+2\rho+|\nabla f|^{2} = 2\rho(1+f)-\frac{n-3}{n-1}R.$$ Observe that condition $(I)$ is trivially satisfied for the Gaussian shrinking Ricci soliton. Also, considering $\mathbb{S}^3$ with $f=0$ we can see that this condition is trivial. For the cylinder $\mathbb{S}^2\times\mathbb{R}$, condition $(I)$ is equivalent to
			$$Ric\leq \dfrac{1}{2+|x|^2}.$$
			Compatible with \cite[Equation 2.20]{Munteanu20170}.\\
			
			\item Moreover, condition $(II)$ holds for Cylinders and for Gaussian solitons. For instance, $\mathbb{S}^3\times\mathbb{R}$ have $\rho=2$ and $R=6$. Thus, the bad term is zero, i.e., $\rho-R/(n-1)=0$. Thus, condition $(II)$ is trivially satisfied. In fact, any shrinking Ricci soliton satisfies $$|\nabla R|^2\leq4R^2|\nabla f|^2.$$ However, for Einstein manifolds condition $(II)$ needs more control. That is, the bad term $\rho-\frac{R}{n-1}=\frac{R}{n}-\frac{R}{n-1}=\frac{-R}{n(n-1)}$ is negative. So, for instance, if $n=3$ we also must have as a hypothesis that ${R}(1+f)\geq \frac{7R^2}{2}.$ If, $R>0$ we can infer as a hypothesis that the potential function must satisfies the inequality $f\geq\frac{7R}{2}-1,$ which is natural considering \eqref{cao2010} and that $R\leq C;$ $C\in\mathbb{R}$ (i.e., $f$ is bound from below).
		\end{enumerate}
	\end{remark}

	By following the same strategy of Theorem \ref{expan} we can prove our next result for expanding Ricci solitons.
	
	\begin{theorem}\label{expan1}
		Let $(M^3,\,g,\,f)$ be a complete gradient expanding Ricci soliton with nonpositive Ricci curvature such that $1\leq -f$ and
		\begin{eqnarray*}\label{asympBene1111}
			|\nabla f|^2|\nabla R|=o(r^{-3});\quad r\to\infty.
		\end{eqnarray*}
		Then, $(M^3,\,g)$ is rotationally symmetric. Moreover, $R$ must be constant.
	\end{theorem}
	
	\begin{remark}
		The hypothesis over $f$, i.e., $1\leq-f$, is reasonable considering the cylinder $\mathbb{R}\times\mathbb{H}^2$ and the Gaussian soliton, see also \eqref{expandf}. We also recommend to the reader \cite[Theorem 30]{tesechan}. In this theorem P.-Y. Chan proved that, an $3$-dimensional complete gradient expanding Ricci soliton such that $\displaystyle\lim_{x\rightarrow+\infty}r(x)^{2}R(x)=0$ must be isometric to $\mathbb{R}^3.$
	\end{remark}
	
	\begin{remark}
		Moreover, the polynomial decay assumed for $|\nabla R|$ is compatible with \cite[Theorem 25]{tesechan}, see also \cite[Theorem 1.2]{deruelle2017} and \cite[Equation 2.20]{Munteanu20170}.
	\end{remark}

	Theorem \ref{expan} also is true for expanding Ricci solitons with the different asymptotic conditions. This condition came from the fact that we do not necessarily have \eqref{cao2010} for expanders.
	\begin{corollary}
		Let $(M^n,\,g,\,f)$ be a complete gradient expanding Ricci soliton such that
		\begin{eqnarray*}
			|\nabla f|^2|\nabla R|=o(r^{-n});\quad r\to\infty,
		\end{eqnarray*}
		and satisfying either $(I)$ or $(II)$.
		Then, $(M^n,\,g)$ is  either locally conformally flat $(n=3,\,4)$ or has harmonic Weyl tensor ($n\geq5$). 
	\end{corollary}
	
	\begin{remark}
		Feldman-Ilmanen-Knopf \cite{feldman2003} constructed a noncompact K\"ahler gradient expander such that it has nonnegative scalar curvature which decays exponentially in $r$.
	\end{remark}

	\section{Background}
	In this section, we shall present some preliminaries that will be useful for the establishment of the desired result.

	It is well-known that a gradient Ricci soliton satisfies the equation 
	\begin{eqnarray}\label{RF1}
		\nabla R=2Ric(\nabla f).
	\end{eqnarray}
	It follows from the above equation that 
	\begin{eqnarray*}\label{Rfcomplet}
		R + |\nabla f|^{2}-2\rho f=C.
	\end{eqnarray*}
	Note that if we normalize $f$ by adding the constant $C$ to it, then we have
	\begin{equation}\label{Rf}
		R + |\nabla f|^{2}=\lambda f,
	\end{equation}
	where $\lambda=2\rho$.
	
	Now, we remember the covariant $3$-tensor $ D_{ijk}$ (cf. \cite{cao2013}). Let the $D$-tensor $D_{ijk}$ be defined by:
	\begin{eqnarray*}
		D_{ijk}&=&\frac{1}{n-2}(R_{jk}\nabla_{i} f-R_{ik}\nabla_{j}f)+\frac{1}{2(n-1)(n-2)}(g_{jk}\nabla_{i}R-g_{ik}\nabla_{j}R)\\
		&&-\frac{R}{(n-1)(n-2)}(g_{jk}\nabla_{i}f-g_{ik}\nabla_{j}f).
	\end{eqnarray*}
	The $D$-tensor is equivalent to the Cotton tensor for three dimensional Ricci solitons (cf. Lemma 3.1 in \cite{cao2013}). In fact, we have $$D_{ijk}=C_{ijk}+W_{ijkl}\nabla^lf.$$ 
	Here, $C$ and $W$ stand for the Cotton and Weyl tensors, respectively.
	
	From a straigthforward computation we can infer that
	
	\begin{eqnarray}\label{Dtensor}
		D_{ijk}&=&\frac{1}{n-2}(R_{jk}\nabla_{i} f-R_{ik}\nabla_{j}f)\nonumber\\
		&&+\frac{1}{2(n-1)(n-2)}[g_{ik}(2R\nabla_{j}f-\nabla_{j}R)-g_{jk}(2R\nabla_{i}f -\nabla_{i}R)].
	\end{eqnarray}
	Consequently, from \eqref{Dtensor} we have the following important lemma (see \cite[Proposition 4]{brendle2011}).
	\begin{lemma}
		Let $(M^n,\,g,\,f)$ be a shrinking (or expanding) gradient Ricci soliton. Then, 
		\begin{eqnarray*}\label{tres}
			(n-2)^{2}|D|^{2} + \frac{|2R\nabla f-\nabla R |^{2}}{2(n-1)}
			&=& |\nabla f|^{2}\langle\nabla R,\,\nabla f\rangle-|\nabla f|^{2}\Delta R\\
			&&+\lambda R |\nabla f|^2 -\frac{1}{2}|\nabla R|^{2}.\nonumber\\
		\end{eqnarray*}
		
	\end{lemma}
	\begin{proof}
		Let us write the norm of $D$ only depending on the function $f$ and the scalar curvature $R$. We begin the computation using (\ref{Dtensor}). To that end, we start with the following identity:
		\begin{eqnarray*}
			|D|^{2} &=& \frac{2}{(n-2)^{2}}|Ric|^{2}|\nabla f|^{2} -\frac{2}{(n-2)^{2}}R_{ik}\nabla_jfR_{jk}\nabla_if \nonumber\\
			&+&\frac{1}{2(n-1)(n-2)^{2}}|2R\nabla f-\nabla R|^{2}\nonumber\\
			&+&\frac{2}{(n-1)(n-2)^{2}}(R_{jk}\nabla_{i} f-R_{ik}\nabla_{j}f)(2R\nabla_{j}f-\nabla_{j}R)g_{ik}.
		\end{eqnarray*}
		
		Then, by using Definition \ref{def1}, \eqref{RF1} and \eqref{Rf} we get
		\begin{eqnarray*}
			|D|^{2} &=&  \frac{2}{(n-2)^{2}}|Ric|^{2}|\nabla f|^{2} -\frac{1}{2(n-2)^{2}}2R_{ik}\nabla_jf2R_{jk}\nabla_if \nonumber\\
			&+&\frac{1}{2(n-1)(n-2)^{2}}|2R\nabla f-\nabla R|^{2}\nonumber\\
			&-&\frac{1}{(n-1)(n-2)^{2}}(2R\nabla_{j}f-\nabla_{j}R)(2R\nabla_{j}f-\nabla_{j}R)\nonumber\\
			&=& \frac{2}{(n-2)^{2}}|Ric|^{2}|\nabla f|^{2} -\frac{1}{2(n-2)^{2}}|\nabla R|^{2}\nonumber\\
			&-&\frac{1}{2(n-1)(n-2)^{2}}|2R\nabla f-\nabla R|^{2}.
		\end{eqnarray*}
		
		Now, we need to use the following identity (cf. \cite[Proposition 2.1]{Eminenti}):
		
		\begin{equation*}
			2|Ric|^{2} =\left<\nabla R, \nabla f\right>+\lambda R-\Delta R
		\end{equation*}
		Therefore, combining this last two identities we obtain the result
		\begin{eqnarray*}
			|D|^{2} + \frac{|2R\nabla f-\nabla R |^{2}}{2(n-1)(n-2)^{2}}
			&=& \frac{|\nabla f|^{2}}{(n-2)^{2}}\langle\nabla R,\,\nabla f\rangle-\frac{|\nabla f|^{2}}{(n-2)^{2}}\Delta R\\
			&&+\frac{\lambda R |\nabla f|^2}{(n-2)^2} -\frac{1}{2(n-2)^{2}}|\nabla R|^{2}.\nonumber\\
		\end{eqnarray*}
		
	\end{proof}
	
	In what follows, we will provide a divergence formula for the shrinking Ricci solitons inspired by \cite{brendle2011,Robinson}. Now, consider a smooth function $\psi:\mathbb{R}\rightarrow\mathbb{R}$ such that $$Y=\nabla R+ \psi(|\nabla f|^2)\nabla f$$ on $M$. Also, from \eqref{Rf} we can rewrite the vector field $Y$ in the form: $$Y=\nabla R+ \psi(\lambda f - R)\nabla f.$$
	
	Our next lemma is a divergence formula inspired by the works of Brendle \cite{brendle2011} and Robinson \cite{Robinson}. Those divergence formulas were the key in classifying steady Ricci solitons and static vacuum spaces, respectively.
	
	\begin{lemma}\label{lema22}
		Let $\big(M^n,\,g,\,f)$ be a shrinking (or expanding) gradient Ricci soliton. Then,
		\begin{eqnarray*}
			2(\lambda f-R)\mathrm{div}\left( Y\right)&=&-2(n-2)^{2}|D|^{2} - \frac{|2R\nabla f-\nabla R |^{2}}{(n-1)}-| \nabla R|^{2}\\
			&&+2(\lambda f- R)[R^{2}-\frac{\lambda}{2}\left(n+2f\right) R+\frac{\lambda^{2}}{2}\left(n+2\right)f].\\
		\end{eqnarray*}
	\end{lemma}
	\begin{proof}
		It is well-known from Definition \ref{def1} that  $\Delta f=\rho n- R$. Hence, 
		\begin{eqnarray*}
			\mathrm{div}(Y) &=& \Delta R+ \lambda \dot{\psi}|\nabla f|^{2} -\dot{\psi}\langle\nabla R,\nabla f\rangle + \psi\Delta f \\
			&=& \Delta R+ \lambda \dot{\psi}|\nabla f|^{2} -\dot{\psi}\langle\nabla R,\nabla f\rangle + (\rho n- R)\psi.
		\end{eqnarray*}
		Thus,
		\begin{eqnarray*}
			-2(\lambda f-R)\mathrm{div}(Y) &=&-2(\lambda f-R)\Delta R- \lambda2(\lambda f-R) \dot{\psi}|\nabla f|^{2}\\
			&&+2(\lambda f-R)\dot{\psi}\langle\nabla R,\nabla f\rangle -2(\lambda f-R) (\rho n- R)\psi.\\
		\end{eqnarray*}
		Combining the above equation with the previous lemma we get
		\begin{eqnarray*}
			2(n-2)^{2}|D|^{2} + \frac{|2R\nabla f-\nabla R |^{2}}{(n-1)}
			&=&-2(\lambda f-R)\Delta R+2(\lambda f-R)\langle\nabla R,\,\nabla f\rangle\nonumber\\
			&&+2\lambda R(\lambda f-R) -|\nabla R|^{2}. \nonumber\\
		\end{eqnarray*}
		Hence,
		\begin{eqnarray*}
			&&2(n-2)^{2}|D|^{2} + \frac{|2R\nabla f-\nabla R |^{2}}{(n-1)}
			\\
			&=&-2(\lambda f-R)\mathrm{div}(Y)+2\lambda(\lambda f-R) \dot{\psi}|\nabla f|^{2} -2(\lambda f-R)\dot{\psi}\langle\nabla R,\nabla f\rangle \nonumber\\ &&+2(\lambda f-R) (\rho n- R)\psi +2(\lambda f-R)\langle\nabla R,\,\nabla f\rangle+2\lambda R(\lambda f-R) -| \nabla R|^{2}.\nonumber
		\end{eqnarray*}
		Consequently,
		\begin{eqnarray*}
			&&2(n-2)^{2}|D|^{2} + \frac{|2R\nabla f-\nabla R |^{2}}{(n-1)}=-2(\lambda f-R)\mathrm{div}(Y)+2\lambda(\lambda f-R) \dot{\psi}|\nabla f|^{2}\\
			&&+2(\lambda f-R) (\rho n- R)\psi +2\lambda R(\lambda f-R) -| \nabla R|^{2}+ 2(\lambda f-R)[1 -\dot{\psi}]\langle\nabla R,\,\nabla f\rangle.
		\end{eqnarray*}
		
		Consider, $\psi$ as an identity function, i.e.,
		\begin{eqnarray*}
			\psi= (\lambda f-R).
		\end{eqnarray*}
		Therefore, 
		\begin{eqnarray*}
			2(\lambda f-R)\mathrm{div}(Y)&=&-2(n-2)^{2}|D|^{2} - \frac{|2R\nabla f-\nabla R |^{2}}{(n-1)}-| \nabla R|^{2}\\
			&&+2(\lambda f-R)^2 (\rho n- R) +2\lambda R(\lambda f-R) +2\lambda(\lambda f-R)|\nabla f|^{2}.
		\end{eqnarray*}
	\end{proof}


	\section{Proof of the main result}
	In this section we present the proof for the main result of this paper.
	\begin{proof}[{\bf Proof of Theorem \ref{thmB}}]
		Now, from Lemma \ref{lema22} we can infer that
		\begin{eqnarray*}
			2(\lambda f-R)\mathrm{div}\left(Y\right)&=&-2(n-2)^{2}|D|^{2} - \frac{|2R\nabla f-\nabla R |^{2}}{(n-1)}-| \nabla R|^{2}\\
			&&+2(\lambda f- R)[R^{2}-\frac{\lambda}{2}\left(n+2f\right) R+\frac{\lambda^{2}}{2}\left(n+2\right)f].\\
		\end{eqnarray*}

		On the other hand, from \eqref{Rf} we have $R\leq 2\rho f=\lambda f$, and since $M$ is compact we can assume that $R\geq\delta=\min_M(R)$. Therefore,
		\begin{eqnarray*}
			R^{2}-\frac{\lambda}{2}\left(n+2f\right) R+\left(n+2\right)\frac{\lambda^2}{2}f \leq \lambda^2f^2 + \left(n+2\right)\frac{\lambda^2}{2}f - \frac{\lambda}{2}\left(n+2f\right)\delta.
		\end{eqnarray*}
		Then, considering 
		\begin{eqnarray*}
			\frac{\delta}{2\lambda}-\frac{n+2}{4}-\sqrt{\left(\frac{n+2}{4}-\frac{\delta}{2\lambda}\right)^2+\frac{n\delta}{2\lambda}}\leq f\leq \frac{\delta}{2\lambda}-\frac{n+2}{4}+\sqrt{\left(\frac{n+2}{4}-\frac{\delta}{2\lambda}\right)^2+\frac{n\delta}{2\lambda}}
		\end{eqnarray*}
		we can infer that
		\begin{eqnarray*}
			\lambda^2f^2 +\left[(n+2)\frac{\lambda^2}{2}-\delta\lambda\right]f - \frac{n\lambda\delta}{2}\leq0
		\end{eqnarray*}
		On the other hand, from \eqref{Rf} we have that $R\leq \lambda f.$ Then, since $R\geq0$ we can infer that the left hand side of the inequality is trivially satisfied.
		
		Now, since $\lambda f-R\geq0$ we can conclude that
		\begin{equation*}
			\mathrm{div}\left(Y\right)\leq0.
		\end{equation*}
		Moreover, 
		\begin{equation*}
			0\geq\int_{M}\mathrm{div}(Y)=0.
		\end{equation*}
		Therefore, $\mathrm{div}(Y)=0$ and 
		\begin{eqnarray*}
			0&=&2(n-2)^{2}|D|^{2} + \frac{|2R\nabla f-\nabla R |^{2}}{(n-1)}+| \nabla R|^{2}\\
			&&-2(\lambda f- R)[\lambda^2f+(\lambda f-R) (n\frac{\lambda}{2}- R)]\geq0.\\
		\end{eqnarray*}
		
		Finally, we can conclude that $(M,\,g,\,f)$ is Einstein.
	\end{proof}

	\begin{proof}[{\bf Proof of Theorem \ref{expan}}]
		Let $\Omega$ be a bounded domain on $M$ with smooth boundary. Using that $Y=\nabla R+(\lambda f-R)\nabla f$ and the divergence theorem, from Lemma \ref{lema22} we get
		\begin{eqnarray*}
			&&\int_{\partial \Omega}\langle(\lambda f-R)Y,\,\nu\rangle = \int_{\Omega}\mathrm{div}\left((\lambda f-R)Y\right) \nonumber\\
			&=&\int_{\Omega}(\lambda f-R)\mathrm{div}\left(Y\right) +\int_{\Omega}\lambda\left<\nabla f,Y\right>- \int_{\Omega}\langle \nabla R,\,Y\rangle\nonumber\\
			&=&\int_{\Omega}\left[-(n-2)^{2}|D|^{2} - \frac{|2R\nabla f-\nabla R |^{2}}{2(n-1)}-\frac{| \nabla R|^{2}}{2}\right.\\
			&&\left.+(\lambda f- R)[R^{2}-\frac{\lambda}{2}\left(n+2f\right) R+\frac{\lambda^{2}}{2}\left(n+2\right)f]\right]\nonumber\\
			&&+\int_{\Omega}\lambda\left<\nabla f,\,\nabla R+(\lambda f-R)\nabla f\right>- \int_{\Omega}\left< \nabla R,\,\nabla R+(\lambda f-R)\nabla f\right>\nonumber\\
			&=&\int_{\Omega}\left[-(n-2)^{2}|D|^{2} - \frac{|2R\nabla f-\nabla R |^{2}}{2(n-1)}-\frac{| \nabla R|^{2}}{2}\right.\\
			&&\left.+(\lambda f-R)^2 (\rho n- R) +\lambda R(\lambda f-R) +\lambda(\lambda f-R)|\nabla f|^{2}\right]\nonumber\\
			&&+\int_{\Omega\cap\{R<\lambda f\}}\lambda(\lambda f-R)|\nabla f|^2+\lambda\left<\nabla f,\nabla R\right>- |\nabla R|^2-(\lambda f-R)\left<\nabla f,\nabla R\right>\nonumber\\
			&=&\int_{\Omega}\left[-(n-2)^{2}|D|^{2} - \frac{|2R\nabla f-\nabla R |^{2}}{2(n-1)}-\frac{3| \nabla R|^{2}}{2}\right.\\
			&&+(\lambda f-R)^2 (\rho n- R) +\lambda R(\lambda f-R)+2\lambda(\lambda f-R)|\nabla f|^{2}\nonumber\\
			&&\left.+\lambda\left<\nabla f,\nabla R\right>-(\lambda f-R)\left<\nabla f,\nabla R\right>\right].\nonumber
		\end{eqnarray*}

		Furthermore, using $(\rho n-R)=\Delta f$ we get
		\begin{eqnarray*}\label{divparalelo}
			(\lambda f-R)^{2}(\rho n-R)&=&\left((\lambda f-R)^{2}\nabla f\right) + 2(\lambda f -R)\left<\nabla R, \nabla f\right> \nonumber\\
			&&-2\lambda (\lambda f-R)|\nabla f|^{2}. \nonumber\\
		\end{eqnarray*}
		Thus,
		\begin{eqnarray*}
			\int_{\partial \Omega}&&\left<(\lambda f-R)\nabla R+(\lambda f-R)^2\nabla f,\nu\right>= \nonumber\\
			&&\int_{\Omega}\left[-(n-2)^{2}|D|^{2} - \frac{|2R\nabla f-\nabla R |^{2}}{2(n-1)}-\frac{3| \nabla R|^{2}}{2}\right.\\
			&& + \left((\lambda f-R)^{2}\nabla f\right) +\lambda R(\lambda f-R) \left.+\lambda\left<\nabla f,\nabla R\right>+(\lambda f-R)\left<\nabla f,\nabla R\right>\right]\nonumber,
		\end{eqnarray*}
		and so
		\begin{eqnarray*}
			\int_{\partial \Omega}(\lambda f -R)\left<\nabla R,\nu\right>  
			&=&\int_{\Omega}\left[-(n-2)^{2}|D|^{2} - \frac{|2R\nabla f-\nabla R |^{2}}{2(n-1)}-\frac{3| \nabla R|^{2}}{2}\right.\\
			&&\left.+\lambda R(\lambda f-R) +\lambda \left<\nabla f,\nabla R\right>+(\lambda f-R)\left<\nabla f,\nabla R\right>\right].\nonumber\\  
		\end{eqnarray*}
		
		Hence, from \eqref{Rf} we get
		\begin{eqnarray*}
			\int_{\partial \Omega}\left<(\lambda f -R)\nabla R,\nu\right>  
			&=&\int_{\Omega}\left[-(n-2)^{2}|D|^{2} - \frac{|2R\nabla f-\nabla R |^{2}}{2(n-1)}-\frac{3| \nabla R|^{2}}{2}\right.\nonumber\\
			&&\left.+\lambda R|\nabla f|^2 +\lambda \left<\nabla f,\nabla R\right>+|\nabla f|^2\left<\nabla f,\nabla R\right>\right].
		\end{eqnarray*}
		By a straightforward computation we obtain
		\begin{eqnarray}\label{corolario1}
			\int_{\partial \Omega}\left<(\lambda f -R)\nabla R,\nu\right>  
			&=&\int_{\Omega}\left[-(n-2)^{2}|D|^{2} \right.\nonumber\\
			&&-\left(\frac{2R}{n-1}-\lambda \right)R|\nabla f|^2+\left(\frac{2R}{n-1}+\lambda \right)\left<\nabla f, \nabla R\right>\nonumber\\
			&&\left.-\left(\frac{3n-2}{2(n-1)}\right)|\nabla R|^2+|\nabla f|^2\left<\nabla f,\nabla R\right>\right].
		\end{eqnarray}
		From \eqref{corolario1} and \eqref{RF1} we get 
		\begin{eqnarray*}
			\int_{\partial \Omega}\left<(\lambda f -R)\nabla R,\nu\right>  
			&=&\int_{\Omega}\left[-(n-2)^{2}|D|^{2} \right.\nonumber\\
			&&-\left(\frac{2R^{2}}{n-1}-\lambda R \right)|\nabla f|^2+2\left(\frac{2R}{n-1}+\lambda \right)Ric(\nabla f,\nabla f)\nonumber\\
			&&\left.-\left(\frac{3n-2}{2(n-1)}\right)|\nabla R|^2+2|\nabla f|^2Ric(\nabla f,\nabla f)\right].
		\end{eqnarray*}
		So, from \eqref{RF1} we get
		\begin{eqnarray}\label{desig}
			\int_{\partial \Omega}\left<(\lambda f -R)\nabla R,\nu\right>  
			&\leq&-\int_{\Omega}(n-2)^{2}|D|^{2}-\int_{\Omega} \left(\frac{3n-2}{2(n-1)}\right)|\nabla R|^2\nonumber\\
			&&+\int_{\Omega}\left[2\left(\frac{2R}{n-1}+\lambda+|\nabla f|^{2} \right)Ric(\nabla f,\nabla f)-\left(\frac{2R}{n-1}-\lambda \right)R|\nabla f|^2\right]\nonumber\\
			&=&-\int_{\Omega}(n-2)^{2}|D|^{2}-\int_{\Omega} \left(\frac{3n-2}{2(n-1)}\right)|\nabla R|^2\nonumber\\
			&&+2\int_{\Omega}\left[\left(2\rho(1+f)-\frac{n-3}{n-1}R\right)Ric(\nabla f,\,\nabla f)-\left(\frac{R}{n-1}-\rho\right)R|\nabla f|^2\right]\nonumber\\
			&\leq&-\int_{\Omega}(n-2)^{2}|D|^{2}-\int_{\Omega} \left(\frac{3n-2}{2(n-1)}\right)|\nabla R|^2\nonumber\\
			&&+\int_{\Omega}\left[\left(2\rho(1+f)-\frac{n-3}{n-1}R\right)|\nabla R||\nabla f|-2\left(\frac{R}{n-1}-\rho\right)R|\nabla f|^2\right].
		\end{eqnarray}
		From the first inequality above we conclude the first item $(I)$ of this theorem.
		
		Then, by considering
		\begin{eqnarray*}
			|\nabla R|&\leq& \frac{(n-1)}{(3n-2)}\left[|\nabla f|\sqrt{\left(2\rho(1+f)-\frac{(n-3)}{(n-1)}R\right)^2+\frac{4(3n-2)}{(n-1)}R\left(\rho-\frac{R}{n-1}\right)}\right]\\
			&-&\frac{(n-1)}{(3n-2)}|\nabla f|\left(2\rho(1+f)-\frac{(n-3)}{(n-1)}R\right)
		\end{eqnarray*}
		we conclude that
		$$- \left[\frac{3n-2}{2(n-1)}\right]|\nabla R|^2+\left(2\rho(1+f)-\frac{n-3}{n-1}R\right)|\nabla f||\nabla R|-2\left(\frac{R}{n-1}-\rho\right)R|\nabla f|^2\leq0.$$

		Then,
		\begin{eqnarray*}
			\int_{\Omega_{\ell}}(n-2)^{2}|D|^{2} &\leq&-\int_{\partial \Omega_{\ell}}\left<(\lambda f -R)\nabla R,\nu\right>.
		\end{eqnarray*}
		Then, making $\ell\to \infty$ we have
		\begin{eqnarray*}
			\int_{M}\left[(n-2)^{2}|D|^{2}\right]\leq\lim_{\ell\to\infty}\int_{\partial\Omega_{\ell}}|\nabla f|^2|\nabla R|=\lim_{\ell\to\infty}\int_{\partial\Omega_{\ell}}(\lambda f - R)|\nabla R|\\
			\leq\lambda\lim_{\ell\to\infty}\int_{\partial\Omega_{\ell}} f|\nabla R| \to 0.
		\end{eqnarray*}
		Considering $|\partial\Omega_\ell|=w_{n-1}r(\ell)^{n-1}$, where $w_{n-1}$ stands for the volume of the $(n-1)$-sphere, from \eqref{cao2010} we get
		\begin{eqnarray*}
			&&\int_{M}\left[(n-2)^{2}|D|^{2}\right]\leq\lambda\lim_{\ell\to\infty}\int_{\partial\Omega_{\ell}} f|\nabla R|\leq w_{n-1}\frac{\lambda}{4}\displaystyle\lim_{r\rightarrow+\infty}r^{n-1}(r+c_2)^2|\nabla R|\to 0.
		\end{eqnarray*}

		Therefore, $D$ vanishes identically and $R$ must be constant. Now, we invoke Corollary 5.1 in \cite{cao2013}.
	\end{proof}

	\begin{proof}[{\bf Proof of Theorem \ref{expan1}}]
		It is well-known that if $Ric\geq0$ (or $Ric\leq0$) we have (cf. \cite[Equation 2.7]{Ni}): 
		\begin{eqnarray*}
			|\nabla R|^2\leq 4R^2|\nabla f|^2.
		\end{eqnarray*}
		In fact,  consider an orthonormal frame $\{e_{1}, e_{2}, e_{3},\ldots,e_{n}\}$ diagonalizing $Ric$ at a regular point $p$, with associated eigenvalues $\lambda_{k}$, $k=1,\ldots, n,$ respectively. That is, $R_{ij}=\lambda_{i}\delta_{ij}$. Since $\nabla R=2Ric(\nabla f)$ we can infer that
		\begin{eqnarray*}
			|\nabla R|^2 = 4\langle Ric(\nabla f),\, Ric(\nabla f)\rangle = 4\displaystyle\sum_{i}\lambda_i^2(\nabla_if)^2\leq4\left(\displaystyle\sum_{i}\lambda_i\right)^2\nabla_if\nabla^if= 4R^2|\nabla f|^2.
		\end{eqnarray*}
		
		Thus, if $R\leq0$ we can conclude that
		$$(|\nabla R|+2R|\nabla f|)(|\nabla R|-2R|\nabla f|)\leq 0.$$
		Hence,
		$$|\nabla R|\leq -2R|\nabla f|.$$ Considering $\frac{n-3}{n-1}R\leq2\rho(1+f)$ from \eqref{Rf} and \eqref{desig} we have
		\begin{eqnarray*}
			\int_{\partial \Omega}\left<(\lambda f -R)\nabla R,\nu\right>  
			&\leq&-\int_{\Omega}(n-2)^{2}|D|^{2}-\int_{\Omega} \left(\frac{3n-2}{2(n-1)}\right)|\nabla R|^2\nonumber\\
			&&+2\int_{\Omega}\left[\left(-2\rho(1+f)+\frac{n-3}{n-1}R\right)-\left(\frac{R}{n-1}-\rho\right)\right]R|\nabla f|^2\nonumber\\
			&=&-\int_{\Omega}(n-2)^{2}|D|^{2}-\int_{\Omega} \left(\frac{3n-2}{2(n-1)}\right)|\nabla R|^2\nonumber\\
			&&-2\int_{\Omega}\left(\rho+|\nabla f|^2  +\frac{3R}{n-1}\right)R|\nabla f|^2\nonumber\\
			&=&-\int_{\Omega}(n-2)^{2}|D|^{2}-\int_{\Omega} \left(\frac{3n-2}{2(n-1)}\right)|\nabla R|^2\nonumber\\
			&&-2\int_{\Omega}\left(\rho+\frac{6\rho}{n-1}f+\frac{n-4}{n-1}|\nabla f|^2\right)R|\nabla f|^2.\nonumber\\
		\end{eqnarray*}

		Assuming $n=3$ and $\rho=-1/2$, we have $1\leq -f.$ We can see that $$-1\leq 1\leq -f\leq 4(-f).$$ Hence,
		\begin{eqnarray*}
			\int_{\partial \Omega}\left<(\lambda f -R)\nabla R,\nu\right>  
			&\leq&-\int_{\Omega}(n-2)^{2}|D|^{2}-\int_{\Omega} \left(\frac{3n-2}{2(n-1)}\right)|\nabla R|^2\nonumber\\
			&&-\int_{\Omega}\left(-1+{3}f-|\nabla f|^2\right)R|\nabla f|^2.\nonumber\\
			&\leq&-\int_{\Omega}(n-2)^{2}|D|^{2}-\int_{\Omega} \left(\frac{3n-2}{2(n-1)}\right)|\nabla R|^2\nonumber\\
			&&-\int_{\Omega}\left(-f-|\nabla f|^2\right)R|\nabla f|^2.\nonumber\\
		\end{eqnarray*}

		Since $R\leq0$, from \eqref{Rf} we get $-f-|\nabla f|^2\leq0.$ Thus,
		\begin{eqnarray*}
			\int_{\partial \Omega}\left<(\lambda f -R)\nabla R,\nu\right>  
			\leq-\int_{\Omega}|D|^{2}-\int_{\Omega} \frac{7}{4}|\nabla R|^2.
		\end{eqnarray*}
		Then,
		\begin{eqnarray*}
			\int_{\Omega_\ell}|D|^{2}+\int_{\Omega_\ell} \frac{7}{4}|\nabla R|^2 &\leq&-\int_{\partial \Omega_{\ell}}\left<(\lambda f -R)\nabla R,\nu\right>.
		\end{eqnarray*}
		Then making, $\ell\to \infty$ we have
		\begin{eqnarray*}
			\int_{M}|D|^{2}+\frac{7}{4}\int_{M} |\nabla R|^2\leq\lim_{\ell\to\infty}\int_{\partial\Omega_{\ell}}(\lambda f-R)|\nabla R|=\lim_{\ell\to\infty}\int_{\partial\Omega_{\ell}} |\nabla f|^2|\nabla R|\\
			= 4\pi\displaystyle\lim_{r\rightarrow+\infty}r^{2}|\nabla f|^2|\nabla R|\to 0,
		\end{eqnarray*}
		where we assumed $|\nabla f|^2$ and $|\nabla R|$ as functions of $r$, at infinity, by hypothesis.

		Therefore, $D$ vanishes identically and $R$ must be constant. Since $D$ is equivalent to the Cotton tensor in three dimensions the result follows (cf. Lemma 3.1 in \cite{cao2013}), i.e., $(M^3,\,g,\,f)$ is locally conformally flat. The rotational symmetry now follows from \cite[Lemma 3.3 and Lemma 4.3]{cao2011}.
	\end{proof}

\end{document}